\documentclass[10pt,a4paper]{amsart} 
\usepackage{url}
\usepackage{amsopn,amsmath,amssymb,amsthm}

\theoremstyle{remark}
\newtheorem{remark}{Remark}
\theoremstyle{plain}
\newtheorem{theorem}[remark]{Theorem}
\newtheorem{proposition}[remark]{Proposition}
\newtheorem{lemma}[remark]{Lemma}

\let\ge=\varepsilon

\newenvironment{notation}{\textbf{Notation.} \upshape}{\hfill\par}

\begin{document}
    \title[A note on Schr\"{o}dinger--Newton systems]{A note on Schr\"{o}dinger--Newton systems with decaying electric potential}
    \author{Simone Secchi}
    \address{Dipartimento di Matematica ed Applicazioni, Universit\`a di Milano--Bicocca, via R.~Cozzi 53, edificio U5, I-20125 Milano (Italy).}
    \email{simone.secchi@unimib.it}

    \date{\today} 
        
    \maketitle
    
    \begin{abstract}
        We prove the existence of solutions for the singularly perturbed Schr\"{o}dinger--Newton system 
        \begin{equation*} 
\left\{
\begin{array}{ll}
  \hbar^2 \Delta \psi - V(x) \psi + U \psi =0  \\
  \hbar^2 \Delta U + 4\pi \gamma |\psi|^2 =0   
\end{array}
\right.
\hbox{in $\mathbb{R}^3$}
\end{equation*}
        with an electric potential \(V\) that decays polynomially fast at infinity. The solution $\psi$ concentrates, as $\hbar \to 0$,  around (structurally stable) critical points of the electric potential. As a particular case, isolated strict extrema of \(V\) are allowed.
    \end{abstract}
    
    \section{Introduction and statement of the result}
    
    Systems of Schr\"{o}dinger--Newton type like
    \begin{equation} \label{eq:1}
\left\{
\begin{array}{ll}
  (2m)^{-1}\hbar^2 \Delta \psi - V(x) \psi + U \psi =0  \\
  \Delta U + 4\pi \gamma |\psi|^2 =0   
\end{array}
\right.
\hbox{in $\mathbb{R}^3$}
\end{equation}
were introduced by R.~Penrose in \cite{P} to describe a system in which a mass point (of mass $m$) is placed at the origin, under the effect of the gravitational field, and interacts with a matter density given by  the square of the wave function $\psi$, which is the solution of the Schr\"{o}dinger equation. In addition, an electric field is generated by a potential $V$. In (\ref{eq:1}), $U$ is the gravitational potential, $\gamma = G m^2$, $G$ is Newton's constant, and $\hbar$ is Planck's constant. 

A few papers have recently appeared in the literature about this (or some related) problem, with particular attention to the so-called \emph{semiclassical regime} \(\hbar \to 0\).
In \cite{WW}, the existence of multibump solutions concentrating (as~\(\hbar \to 0\)) at local minimum points of $V$ is proved, under the assumption that 
\begin{equation*}
\inf_{x \in \mathbb{R}^3} V(x)>0.
\end{equation*}
This is a broadly--used assumption even in the study of a singular Schr\"{o}dinger equation of the form
\begin{equation} \label{eq:3}
-\hbar^2 \Delta u + V(x)u=f(u).
\end{equation}
Many efforts have been made to prove existence  results  (both for $\hbar=1$ and $\hbar \to 0$) for (\ref{eq:3}) when the potential $V$ is either null in some region (\cite{BW}), or tends to zero at infinity (\cite{AFM,MVS}), or even both things (\cite{BW2}). In this note, we are interested in the second case, under some additional assumptions on the rate of decay at infinity. It will be convenient to scale variables in (\ref{eq:1}):
\begin{displaymath}
\psi = \frac{1}{\hbar} \frac{\hat{\psi}}{\sqrt{8 \pi \gamma m}}, \quad V = \frac{1}{2m} \hat{V}, \quad U = \frac{1}{2m} \hat{U}.
\end{displaymath}
Then (\ref{eq:1}) becomes
\begin{equation} \label{eq:4}
\left\{
\begin{array}{ll}
  \hbar^2 \Delta \hat{\psi} - \hat{V}(x) \hat{\psi} + \hat{U} \hat{\psi} =0  \\
  \hbar^2 \Delta \hat{U} + 4\pi \gamma |\hat{\psi}|^2 =0   
\end{array}
\right.
\hbox{in $\mathbb{R}^3$}.
\end{equation}
By classical potential theory, we can solve the second equation of (\ref{eq:4}) with respect to $U$, and plug the result into the first equation. We need to solve the single \emph{non-local} equation
\begin{equation*}
\hbar^2 \Delta \hat{\psi} - \hat{V}(x)\hat{\psi} + \frac{1}{4 \pi \hbar^2} \left( \int_{\mathbb{R}^3} \frac{|\hat{\psi}(\xi)|^2}{|x-\xi |}\, d\xi \right) \hat{\psi}=0 \quad \hbox{in $\mathbb{R}^3$}.
\end{equation*}
We set for typographical convenience~$\varepsilon=\hbar$, $u=\hat{\psi}$, and $W(x)=1/(4 \pi |x|)$. We will study the singularly perturbed equation
\begin{equation} \label{eq:5}
-\varepsilon^2 \Delta u + V(x) u = \frac{1}{\varepsilon^2} \left( W * |u|^2 \right) u \quad \hbox{in $\mathbb{R}^3$}.
\end{equation}
\begin{remark}
Formally, the local nonlinearity $u^3$ corresponds to the convolution kernel $W=\delta$, the Dirac delta distribution.
\end{remark}
We will retain the following assumptions on $V$:
\begin{itemize}
\item[(V1)] there exist constants $A_0$, $A_1 >0$ and there exists $\alpha \in [0,2]$, such that
\begin{equation*}
\frac{A_0}{1+|x|^\alpha} \leq V(x) \leq A_1 \quad \hbox{for all $x \in \mathbb{R}^3$};
\end{equation*}
\item[(V2)] there exists a constant $V_1>0$ such that 
\begin{equation*}
\sup_{x \in \mathbb{R}^3} |\nabla V(x)| \leq V_1.
\end{equation*}
\end{itemize}
The second assumption is rather technical, while the first one allows the potential $V$ to decay to zero at infinity, although not faster than polynomially.

A simple change of variables yields
\begin{equation}\label{eq:7}
-\Delta u + V(\varepsilon x) u = \left( W * |u|^2 \right) u \quad \hbox{in $\mathbb{R}^3$}.
\end{equation}
We remark that the homogeneity of $W$ produces cancellation of \(\ge\) in the right-had side of (\ref{eq:7}). 
\begin{remark}
A similar model has been studied in \cite{MN}, with a similar technique. However, the equation of \cite{MN} does not scale coherently in the non-local term. In some sense, the Authors study a system like (\ref{eq:4}), in which the second equation is not singularly perturbed.
On the other hand, they allow very general convolution kernels $W$, even without homogeneity, but need a \emph{periodic potential} $V$. However, we do not know any (non-trivial) explicit convolution kernel that matches all the requirements of \cite{MN}, in particular the \emph{non-degeneracy} of the limiting equation (see Theorem \ref{th:2} for a definition of non-degeneracy in our context).
\end{remark}
\begin{remark}
In the recent paper \cite{CSS}, the Authors have studied a non-local equation like (\ref{eq:5}), under the action of an external magnetic field. The technique is based on a compactness analysis of the set of ground--states for the limiting problem. We do not believe that this method can be directly applied to the case under consideration.
\end{remark}
We can state our main existence result. 

\begin{theorem} \label{th:main}
Under the assumption (V1) and (V2), suppose that $V$ is smooth and that $x_0\in \mathbb{R}^3$ is an isolated strict local minimum (or maximum) point of $V$. Then for every $\varepsilon$ small enough, there exists (at least) a solution $u_\varepsilon$ of (\ref{eq:7}) such that $\int_{\mathbb{R}^3} |\nabla u_\varepsilon|^2 + V(\varepsilon x)|u_\varepsilon|^2 < \infty$. This solution is positive and concentrates at $x_0$ as $\varepsilon \to 0$.
\end{theorem}

\begin{remark}
Our solution is not, in general, a ground-state solution. Even for a single Schr\"{o}dinger equation with decaying potential, ground-state solutions need not exist at all (see \cite{AFM}). However, it is a \emph{bound--state} solution, since $\int_{\mathbb{R}^3} |\nabla u_\varepsilon|^2 + V(\varepsilon x)|u_\varepsilon|^2 < \infty$.
\end{remark}

\begin{remark}
It is easy to show that if a bound-state solution concentrates at some point $x_0 \in \mathbb{R}^3$, then $\nabla V(x_0)=0$. This can be done as in \cite{ABC} or \cite{SS}.
\end{remark}

\begin{remark}
 More general existence results can be proved. In particular, we can allow $V$ to possess a whole (compact) manifold of non--degenerate critical points. We will state this result in the last section of the paper.
\end{remark}

We spend a few words about the technique we employ to solve (\ref{eq:7}). This equation has a variational structure, so that we may use tools from Critical Point Theory. The elegant approach of \cite{AFM} does not seem to apply, since it is based on fine inequalities for the \emph{local} nonlinearity of the equation in order to recover compactness.
Similar non--local inequalities are not available, as far as we know. The penalization method introduced in \cite{BVS,MVS} (a refinement of the celebrated penalization technique introduced in \cite{DPF}) consists of changing the nonlinearity outside some bounded region, and again this seems quite hard to adapt, since in our equation the right--hand side depends on the values of the unknown itself \emph{over the whole space}.

The most flexible approach seems to be based on a perturbation technique in Critical Point Theory, introduced and widely applied by Ambrosetti \emph{et al.} some years ago. We refer to the book \cite{AM} for the details. In particular, we will modify the scheme of \cite{AMS} and \cite{AMR} to adapt it to our equation. Loosely speaking, we will (1) construct a set of ``quasi-solutions'', and (2) construct around this set a \emph{finite--dimensional natural constraint}: our problem will be therefore reduced to the search for critical points of a function of three variables. Some difficulties will arise, due to the fact that the potential function $V$ can approach zero at infinity.

Indeed, if $V$ is bounded away from zero, one can work in the usual Sobolev space $H^1(\mathbb{R}^3)$. This is what is done in \cite{WW}. Since $V$ is allowed to decay to zero, the integral $\int_{\mathbb{R}^3} V |u|^2$ need not be finite for every $u \in H^1(\mathbb{R}^3)$. It would be natural to work in some weighted space, but then the non-local term may be undefined. This kind of trouble already arises in single Schr\"{o}dinger equations, see \cite{AFM}.
We will adapt an idea introduced in \cite{AMR} (see also \cite{AR}): a suitable truncation \emph{inside} the convolution will be performed, in such a way that a new functional will be defined smoothly on the natural weighted Sobolev space. We stress that this truncation is not used in order to confine solutions in any region of the physical space. The concentration behavior of our solutions as $\varepsilon \to 0$ is \emph{not} a consequence of this penalization (as it is in the scheme of \cite{BVS}, for example), but follows from the construction itself.

\begin{notation}
\begin{enumerate}
\item $D^{1,2}(\mathbb{R}^3)$ is the space of all measurable functions $u \colon \mathbb{R}^3 \to \mathbb{R}$ such that $\int_{\mathbb{R}^3} |\nabla u|^2 \, dx < \infty$. This space is continuously imbedded into $L^6(\mathbb{R}^3)$, since $6$ is the Sobolev critical exponent in dimension three.
\item To save space, we will often write $\int$ instead of $\int_{\mathbb{R}^3}$ (and also instead of $\int_{\mathbb{R}^3 \times \mathbb{R}^3}$).
\item The letter $C$ will denote (positive) constants whose precise value is not made explicit, and that may vary from line to line. Inside a proof, we may also write $c_1$, $c_2$, $c_3$, \ldots, to denote possibly different constants.
\item The symbol $\nabla$ and $\nabla^2$ will be freely used to denote first and second derivatives of functions and functionals, especially when derivatives are thought of as vectors (through the Riesz isomorphism). 
\item Sometimes we will use Landau's symbols: $o_\ge(1)$ will stand, for example, for a quantity that tends to zero as $\ge \to 0$.
\end{enumerate}
\end{notation}

\section{The variational setting and background results}

As announced in the Introduction, we will work in the weighted space
\begin{equation*}
E=E_\varepsilon = \left\{ u \in D^{1,2}(\mathbb{R}^3) \mid \int_{\mathbb{R}^3} V(\varepsilon x) |u(x)|^2 \, dx < +\infty \right\},
\end{equation*}
which is a Hilbert space with respect to the inner product
\begin{equation*}
\langle u,v \rangle = \langle u,v \rangle_E = \int_{\mathbb{R}^3} \nabla u \cdot \nabla v \, dx + \int_{\mathbb{R}^3} V(\varepsilon x) u(x) v(x) \, dx
\end{equation*}
and corresponding norm
\begin{equation*}
\|u\|^2=\|u\|_E^2 = \int_{\mathbb{R}^3} |\nabla u|^2 \, dx + \int_{\mathbb{R}^3} V(\varepsilon x) |u(x)|^2 \, dx.
\end{equation*}

\bigskip

In the sequel, we will make an extensive use of the following \emph{Hardy--Littlewood--Sobolev inequality}. We refer to \cite{LL,S} for a proof, in a more general setting.

\begin{proposition} \label{prop:HLS}
Let $f \in L^{p}(\mathbb{R}^3)$, $g \in L^{t}(\mathbb{R}^3)$, with $\frac{1}{p}+\frac{1}{t}+\frac{1}{3}=2$. Then 
\begin{equation*}
\left| \int_{\mathbb{R}^3 \times \mathbb{R}^3} \frac{f(x)g(y)}{|x-y|}\, dx \, dy \right| \leq C \|f\|_p \|g\|_t
\end{equation*}
for some universal constant $C>0$ independent of $f$ and $g$. Moreover, if $r$, $s\in (1,+\infty)$ and $0 < \lambda < 3$ satisfy $\frac{1}{r}+\frac{\lambda}{3}=1+\frac{1}{s}$, then
\begin{equation} \label{eq:hls6}
\left\| \int_{\mathbb{R}^3} \frac{g(y)}{|\cdot -y|}\,  dy \right\|_s \leq C \|g\|_r
\end{equation}
\end{proposition}
For $u \in E$, the double integral $\int_{\mathbb{R}^3 \times \mathbb{R}^3} W(x-y)u(x)^2u(y)^2 \, dx \, dy$ may be infinite. To exploit the variational structure of (\ref{eq:7}), we introduce a smooth function $\mu \in C^\infty(\mathbb{R})$ such that $0 \leq \mu \leq 1$, $\mu(s)=1$ if $s < 1$, and $\mu (s)=0$ if $s>2$. For any $u \in E$, we set
\[
\Upsilon (x,u) = \mu \left( \bar{c}^{-1} |u| \cdot (1+\varepsilon |x|)^\theta \right), 
\]
where the constants $\bar{c}$ and $\theta$ will be chosen appropriately later. Then, we define the truncation
\begin{equation*}
T_\varepsilon u (x) = \Upsilon(x,u)u+ \left( 1-\Upsilon (x,u) \right) \frac{\bar{c}}{(1+\varepsilon |x|)^\theta}.
\end{equation*}
In particular,
\[
T_\varepsilon u (x) =
\begin{cases}
u(x), &\hbox{if } u(x) < \bar{c} \left(1+\varepsilon |x| \right)^{-\theta} \\
\bar{c} \left(1+\varepsilon |x| \right)^{-\theta} &\hbox{if } u(x) > 2\bar{c} \left(1+\varepsilon |x| \right)^{-\theta}
\end{cases}
\]
and we can always choose $\theta>0$ so that the functional $I_\varepsilon \colon E \to \mathbb{R}$ defined by
\begin{equation*}
I_\varepsilon (u) = \frac{1}{2} \|u\|_E^2 - \frac{1}{4} \int_{\mathbb{R}^3 \times \mathbb{R}^3} W(x-y) |T_\varepsilon u(x)|^2 |T_\varepsilon u(y)|^2\, dx \, dy
\end{equation*}
is of class $C^1(E)$. As such, every critical point $u_\varepsilon$ of $I_\varepsilon$ such that $|u_\varepsilon (x)| < \bar{c} \left( 1+\varepsilon |x| \right)^{-\theta}$ is a solution of (\ref{eq:7}). 

To proceed further, we need to recall some results about  (\ref{eq:7}) with $\varepsilon$ fixed.

\begin{theorem} \label{th:2}
There exists a unique radial solution of the problem
\begin{equation*}
\begin{cases}
-\Delta U + U = \left( W * |U|^2 \right) U &\hbox{in $\mathbb{R}^3$} \\
U>0 &\hbox{in $\mathbb{R}^3$} \\
U(0)=\max_{x\in\mathbb{R}^3} U(x).
\end{cases}
\end{equation*}
The solution~$U$ is a strictly decreasing function that decays exponentially fast at infinity together with its first derivatives. Moreover, it is non-degenerate in the following sense: if $\phi \in H^2(\mathbb{R}^3)$ solves the linearized equation 
\[
-\Delta \phi + \phi = \left( W * |U|^2 \right) \phi + 2 \left( W * (U \phi) \right) U,
\]
then $\phi$ is a linear combination of $\partial U / \partial x_j$, $j=1,2,3$.
\end{theorem}
\begin{proof}
The existence of $U$ is proved in \cite{MT}. The second part of the Theorem is proved in \cite{WW}.
\end{proof}
\begin{remark} \label{rem:9}
Since the unique solution $U$ of Theorem \ref{th:2} has a mountain--pass characterization, it is not hard to refine the statement as follows: the Euler--Lagrange functional associated to the equation solved by $U$ has a second derivative that is negative--definite along the direction $U$, and positive--definite on the orthogonal of the subspace defined by the (three) partial derivatives of $U$ itself.
\end{remark}
\begin{remark}
The exponential decay of the solution $U$ will be used to get high--order summability when needed. Many integral estimates would be more involved (and even wrong), if $U$ had just a power-like decay at infinity. We refer to the comments at the end of the last section.
\end{remark}
For every $a>0$, the function $U_a(x)=a U(\sqrt{a}\, x)$ is the unique solution of
\begin{equation*}
-\Delta U_a + a U_a = \left( W * |U_a|^2 \right) U_a \quad \hbox{in $\mathbb{R}^3$}.
\end{equation*}
Given $\xi \in \mathbb{R}^3$, we set
\begin{equation*}
z_{\varepsilon, \xi}(x) = U_{V(\varepsilon \xi)} (x-\xi) = V(\varepsilon \xi) U \left( \sqrt{V(\varepsilon \xi)} \, (x-\xi) \right).
\end{equation*}
Hence $z_{\varepsilon,\xi}$ solves the \emph{limiting equation}
\begin{equation}\label{eq:16}
-\Delta z_{\varepsilon,\xi} + V(\varepsilon \xi) z_{\varepsilon,\xi} = \left( W * |z_{\varepsilon,\xi}|^2 \right) z_{\varepsilon,\xi} \quad \hbox{in $\mathbb{R}^3$}.
\end{equation}
Although (\ref{eq:7}) cannot be set in $H^1(\mathbb{R}^3)$ because $V$ can be arbitrarily small, not all is lost: $z_{\ge,\xi}$ is \textit{almost} a critical point of $I_\ge$. This is rather standard if $V$ is far from zero (see \cite{AM,AMS}), but it is still true under our assumptions. See Lemma \ref{lem:10} below.

Therefore, we will look for a solution of (\ref{eq:7}) near the manifold
\begin{equation*}
Z=Z_\varepsilon=\left\{ z_{\varepsilon, \xi} \mid \xi \in \mathbb{R}^3, \ |\varepsilon \xi| < 1 \right\}.
\end{equation*}
To this aim, we will perform a finite-dimensional reduction.
If $P=P_{\ge,\xi}$ stands for the orthogonal projection of $E$ onto $(T_{z_{\varepsilon,\xi}} Z )^\perp$, we will first solve the auxiliary equation
\begin{equation} \label{eq:17}
P \nabla I_\varepsilon (z_{\varepsilon,\xi}+w)=0
\end{equation}
with respect to $w=w(\ge,\xi)$ in some suitable subspace orthogonal to $Z$. Since we can prove that $\tilde{Z} = \{z_{\ge,\xi}+w \}$ is a natural constraint for $I_\varepsilon$, we will show that an isolated local minimum or maximum point $x_0$ of $V$ generates a critical point of the constrained functional $I_\varepsilon$ on $\tilde{Z}$. Thus our main existence result will be proved. This perturbation method is a modification of that introduced in \cite{AMS}, and follows closely \cite{AMR}. Since the only difference with \cite{AMR} lies in the convolution term, we will often be sketchy on those arguments that do not involve it, and refer to \cite{AMR} for the details.

\section{Solving the auxiliary equation}

We recall an analytic estimate on $V$. We omit the proof, since it appears in \cite{AMR}.

\begin{lemma} \label{lem:1}
Let $\alpha>0$, suppose that $V(x)>a |x|^{-\alpha}$ for $|x|>1$, and let $m>0$ be given. Then there exist $\varepsilon_0>0$ and $R>0$ such that
\[
V(\varepsilon x + y) \geq \frac{m}{|x|^\alpha}
\]
whenever $|x|\geq R$, $\varepsilon \leq \varepsilon_0$, $y \in \mathbb{R}^3$ and $|y| \leq 1$.
\end{lemma}

Next, we prove that $Z$ is a manifold of \textit{quasi-solutions} of (\ref{eq:7}). We begin with a useful fact that will be systematically exploited hereafter. Since $z_{\varepsilon,\xi}$ decays exponentially fast at infinity, we can (and do) choose $\bar{c}$ and $\theta$ in such a way that $T_\varepsilon z_{\varepsilon,\xi}=z_{\varepsilon,\xi}$. Therefore, we can simply write $z_{\varepsilon,\xi}$ when we work with $I_\varepsilon$ on $Z$.

\begin{lemma} \label{lem:10}
There is a constant $\bar{C}>0$ such that $\|\nabla I_\varepsilon (z_{\varepsilon,\xi})\| \leq \bar{C} \varepsilon$ whenever $|\varepsilon \xi| \leq 1$.
\end{lemma}
\begin{proof}
Fix any $v \in E$. Then
\begin{multline*}
\nabla I_\varepsilon (z_{\varepsilon,\xi})v = \int \left( \nabla z_{\varepsilon,\xi} \cdot \nabla v + V(\varepsilon x) z_{\varepsilon,\xi} v \right) \\
{}- \int W(x-y) z_{\varepsilon,\xi}(x)^2 z_{\varepsilon,\xi}(y) v(y) \, dx \, dy
\end{multline*}
and from
\[
-\Delta z_{\varepsilon,\xi} + V(\varepsilon \xi) z_{\varepsilon,\xi} = \left( W*|z_{\varepsilon,\xi}|^2 \right) z_{\varepsilon,\xi}
\]
we get
\[
\left| \nabla I_\varepsilon (z_{\varepsilon,\xi})v \right| \leq \left| \int \left[ V(\varepsilon x) - V(\varepsilon \xi) \right] 
z_{\varepsilon,\xi}(x)^2 v(x) \, dx \right| .
\]
We have cleared the convolution term, and the proof reduces now to the arguments and the estimates of \cite[Lemma 3]{AMR}: we apply the Intermediate Value Theorem to the difference $V(\varepsilon x) - V(\varepsilon \xi)$, and we exploit the exponential decay of $z_{\ge,\xi}$ to get the necessary integrability. We omit the details.
\end{proof}

\begin{remark}
 The homogeneity of the non-local term (in the double-integral form) will often simplify our calculations. In this respect, the local nonlinearity may be harder to treat than our convolution term.
\end{remark}

The next Lemma takes care of another technical ingredient of the perturbation technique.

\begin{lemma} \label{lem:3}
The operator $\nabla^2 I_\varepsilon (z_{\varepsilon,\xi}) \colon E \to E$ is a compact perturbation of the identity operator.
\end{lemma}
\begin{proof}
The second derivative $\nabla^2 I_\varepsilon (z_{\varepsilon,\xi})$ acts on generic elements $v$, $w \in E$ as
\begin{equation} \label{eq:nabla2}
\langle \nabla^2 I_\varepsilon (z_{\varepsilon,\xi})v,w\rangle = \int \left( \nabla v \cdot \nabla w + V(\varepsilon x) vw \right) \, dx - K(v,w),
\end{equation}
where
\begin{multline} \label{eq:nabla2bis}
K(v,w)=\int W(x-y) z_{\varepsilon,\xi}(x)^2 v(x) w(y) \, dx \, dy \\
{}+2 \int W(x-y)  z_{\varepsilon,\xi}(y) v(y) z_{\varepsilon,\xi}(x) w(x) \, dx \, dy.
\end{multline}
We need to prove that $K$ is a compact operator. Take any couple of sequences $\{v_m\}$, $\{w_m\}$ in $E$, such that  $\|v_m\| \leq 1$, $\|w_m\| \leq 1$, $v_m \rightharpoonup v_0$ and $w_m \rightharpoonup w_0$ (weak convergence in $E$). Without loss of generality, we assume $v_0=0=w_0$. We are going to show that $K(v_m,w_m) \to 0$. Consider first
\begin{equation*}
\left| \int W(x-y)  z_{\varepsilon,\xi}(y) v(y) z_{\varepsilon,\xi}(x) w(x) \, dx \, dy \right| \leq C \|z_{\varepsilon,\xi} v_m\|_{\frac{6}{5}} \|z_{\varepsilon,\xi}w_m\|_{\frac{6}{5}}.
\end{equation*}
Here we have chosen $p=t$ in the Hardy--Littlewood--Sobolev inequality. Now,
\begin{equation}
\|z_{\varepsilon,\xi} v_m\|_{\frac{6}{5}}^{\frac{6}{5}} = \int_{|x-\xi|<R} |z_{\varepsilon,\xi}(x) v_m(x)|^{\frac{6}{5}}\, dx + \int_{|x-\xi|\geq R} |z_{\varepsilon,\xi}(x) v_m(x)|^{\frac{6}{5}}\, dx.
\end{equation}
Since $v_m \in H^1 (B(\xi,R))$, we have $v_m \to 0$ strongly in $L^2(B(\xi,R))$. The H\"{o}lder inequality implies then
\begin{multline*}
\int_{|x-\xi|<R} |z_{\varepsilon,\xi}(x) v_m(x)|^{\frac{6}{5}}\, dx \leq \\
\left( \int_{|x-\xi|<R} |v_m(x)|^2 \, dx
\right)^\frac{3}{5}
\left(
\int_{|x-\xi|<R} |z_{\varepsilon,\xi}(x)|^3 \, dx
\right)^\frac{2}{5} 
= o(1)
\end{multline*}
as $m \to +\infty$. To estimate the integral outside the ball $B(\xi,R)$, pick $\eta>0$: there exists $R>0$ such that
\begin{multline*}
\int_{|x-\xi|\geq R} |z_{\varepsilon,\xi}(x) v_m(x)|^{\frac{6}{5}}\, dx \leq \\
\left( \int_{|x-\xi|\geq R} |v_m(x)|^6 \, dx
\right)^\frac{1}{5}
\left(
\int_{|x-\xi|\geq R} |z_{\varepsilon,\xi}(x)|^\frac{3}{2} \, dx
\right)^\frac{4}{5} 
\leq C \eta
\end{multline*}
thanks to the exponential decay of $z_{\varepsilon,\xi}$ and the boundedness of $\{v_m\}$. The term with $w_m$ is perfectly symmetric, and in conclusion
\begin{displaymath}
\limsup_{m \to +\infty} \left| \int W(x-y)  z_{\varepsilon,\xi}(y) v(y) z_{\varepsilon,\xi}(x) w(x) \, dx \, dy \right| \leq C \eta^\frac{5}{6}.
\end{displaymath}
By repeating these arguments, we can show that
 \begin{displaymath}
\limsup_{m \to +\infty} \left| \int W(x-y) z_{\varepsilon,\xi}(x)^2 v(x) w(y) \, dx \, dy \right| \leq C \eta^\frac{5}{6},
\end{displaymath}
and since $\eta>0$ is arbitrary, the proof is complete.
\end{proof}
\begin{remark}
 The argument contained in the proof just finished, and based on Proposition \ref{prop:HLS}, will be used several times in the sequel. 
\end{remark}
We come now to the most important step of this section: the invertibility of the gradient $\nabla I_\varepsilon (z_{\varepsilon,\xi})$ along orthogonal directions to the manifold $Z$. This will be of crucial importance when trying to solve the auxiliary equation (\ref{eq:17}) with respect to $w$.

The following Lemma is technical, and we recall it here for the sake of completeness. Its proof can be found in \cite{AMS}, Claim 1 inside the proof of Lemma 5.
\begin{lemma} \label{lem:tech1}
There exist $R \in (\sqrt[4]{\varepsilon},\sqrt{\varepsilon})$ such that
\[
\int_{R<|x-\xi|<R+1} |\nabla v|^2 + |v|^2 < C \sqrt{\varepsilon} \|v\|^2
\]
for every $v \perp \operatorname{span} \left\{
z_{\varepsilon,\xi},\frac{\partial z_{\varepsilon,\xi}}{\partial \xi_1},\frac{\partial z_{\varepsilon,\xi}}{\partial \xi_2}, \frac{\partial z_{\varepsilon,\xi}}{\partial \xi_3}
\right\}$.
\end{lemma}

\begin{lemma}
For some $C'>0$ and any $\varepsilon \ll 1$, $P\nabla I_\varepsilon (z_{\varepsilon,\xi})$ is uniformly invertible for every $\xi \in \mathbb{R}^3$ such that $|\varepsilon \xi| \leq 1$ and $\|\left( P\nabla I_\varepsilon (z_{\varepsilon,\xi}) \right)^{-1}\| \leq C'$.
\end{lemma}
\begin{proof}
By Lemma \ref{lem:3}, we need to prove that in a suitable neighborhood of zero there fall no eigenvalues of $P\nabla I_\varepsilon (z_{\varepsilon,\xi})$. The proof is rather technical, and we will write down only those steps that require significant modifications to the arguments of \cite[Lemma 5]{AMR}. To begin with,  we recall that (see (\ref{eq:nabla2}) and (\ref{eq:nabla2bis}))
\begin{eqnarray*}
\langle \nabla^2 I_\varepsilon (z_{\varepsilon,\xi})v,w \rangle & = & \int \nabla v \cdot \nabla w + V(\varepsilon x) vw \notag \\
&& {-} 2 \int W(x-y)  z_{\varepsilon,\xi}(x) w(x) z_{\varepsilon,\xi}(y) v(y) \, dx \, dy \notag\\
&& {-} \int W(x-y) z_{\varepsilon,\xi}(x)^2 w(y) v(y) \, dx\, dy.
\end{eqnarray*}
Choosing $v=w=z_{\varepsilon,\xi}$ and recalling that $z_{\varepsilon,\xi}$ solves (\ref{eq:16}), we get
\begin{multline*}
\langle \nabla^2 I_\varepsilon (z_{\varepsilon,\xi})v,w \rangle = \int \left[ V(\varepsilon x) - V(\varepsilon \xi) \right] z_{\varepsilon,\xi}^2 \\
{}- 2 \int W(x-y) z_{\varepsilon,\xi}(x)^2
z_{\varepsilon,\xi}(y)^2 \, dx\, dy \leq -C \|z_{\varepsilon,\xi}\|^2 < 0.
\end{multline*}
Let us abbreviate
\begin{displaymath}
X=\operatorname{span}\left\{
z_{\varepsilon,\xi},\frac{\partial z_{\varepsilon,\xi}}{\partial \xi_1},\frac{\partial z_{\varepsilon,\xi}}{\partial \xi_2}, \frac{\partial z_{\varepsilon,\xi}}{\partial \xi_3}
\right\}.
\end{displaymath}
Again, by the exponential decay of $\frac{\partial z_{\varepsilon,\xi}}{\partial \xi_1}$ (see Theorem \ref{th:2}), $X \subset H^1(\mathbb{R}^3)$. The rest of the proof consists in showing the inequality
\begin{equation}
\label{eq:21}
\langle \nabla^2 I_\varepsilon (z_{\varepsilon,\xi})v,v\rangle \geq C \|v\|^2 \quad\hbox{for all $v \perp X$}.
\end{equation}
We will assume that $\|v\|=1$.
Next, for $R$ as in Lemma \ref{lem:tech1},  we select a cut-off function $\chi_R \colon \mathbb{R}^3 \to \mathbb{R}$ such that $\chi_R =1$ in $B(\xi,R)$, $\chi_R =0$ in $\mathbb{R}^3 \setminus B(\xi,R+1)$, and $|\nabla \chi_R|<2$. We decompose $v=v_1+v_2$, where $v_1=\chi_R v$ and $v_2 = (1-\chi_R)v$. As in \cite{AMR}, we can easily check that Lemma \ref{lem:tech1} implies
\begin{align*}
\langle v_1,v_2 \rangle &= O(\sqrt{\varepsilon}) \\
\|v_1\|^2 &= \int_{B(\xi,R+1)} |\nabla v_1|^2 + V(\ge x) |v_1|^2 + O (\sqrt{\ge}) \\
\|v_2\|^2 &= \int_{\mathbb{R}^3 \setminus B(\xi,R)} |\nabla v_2|^2 + V(\ge x) |v_2|^2 +  O (\sqrt{\ge}).
\end{align*}

To prove (\ref{eq:21}), we exploit the bilinearity of the second derivative, and split it as follows:
\begin{multline*}
\langle \nabla^2 I_\varepsilon (z_{\varepsilon,\xi})v,v\rangle = \langle \nabla^2 I_\varepsilon (z_{\varepsilon,\xi})v_1,v_1\rangle + \langle \nabla^2 I_\varepsilon (z_{\varepsilon,\xi})v_2,v_2\rangle \\
{}+2\langle \nabla^2 I_\varepsilon (z_{\varepsilon,\xi})v_1,v_2\rangle .
\end{multline*}
As in the proof of Lemma \ref{lem:3}, we can estimate
\begin{multline} \label{eq:23}
\left| \langle \nabla^2 I_\varepsilon (z_{\varepsilon,\xi})v_1,v_2\rangle \right| \leq \left| \langle v_1,v_2 \rangle \right| \\
{} + C \left| \int_{B(\xi,R+1)\setminus B(\xi,R)} W(x-y) z_{\varepsilon,\xi}(x) v(x) z_{\varepsilon,\xi} (y) v(y)\, dx\, dy \right| \\
\leq C \sqrt{\varepsilon}.
\end{multline}
The ``far-away'' term $\langle \nabla^2 I_\varepsilon (z_{\varepsilon,\xi})v_2,v_2\rangle$ is easier to study than in \cite{AMR}, due to the homogeneity of the convolution term. Indeed,
\begin{multline} \label{eq:24}
\left| \langle \nabla^2 I_\varepsilon (z_{\varepsilon,\xi})v_2,v_2\rangle \right| = \int_{\mathbb{R}^3 \setminus B(\xi,R)} |\nabla v_2|^2 + V(\varepsilon x) |v_2|^2 \\
{}- 2 \int W(x-y) z_{\varepsilon,\xi}(x)v_2(x) z_{\varepsilon,\xi}(y)v_2(y)\, dx \, dy \\
{}- \int W(x-y) z_{\varepsilon,\xi}(x)^2 v_2(x) v_2(y)\, dx \, dy \\
\geq \int_{\mathbb{R}^3 \setminus B(\xi,R)} |\nabla v_2|^2 + V(\varepsilon x) |v_2|^2 - C(R) \left( \int_{\mathbb{R}^3 \setminus B(\xi,R)} |\nabla v_2|^2 \right)^\frac{2}{5}\\
= \int_{\mathbb{R}^3 \setminus B(\xi,R)} |\nabla v_2|^2 + V(\varepsilon x) |v_2|^2 + o_\varepsilon (\|v_2\|^2) \\
\geq \frac{1}{2} \|v_2\|^2 +  O (\sqrt{\ge}).
\end{multline}
Again, we have used the Hardy--Littlewood--Sobolev inequality and the properties of $z_{\varepsilon,\xi}$. The last term to estimate is the one ``concentrated'' on $B(\xi,R)$. Now,
\begin{multline*}
\langle \nabla^2 I_\varepsilon (z_{\varepsilon,\xi})v_1,v_1 \rangle = \|v_1\|^2 -\int W(x-y) z_{\varepsilon,\xi}(x)^2 v_1(y)^2 \, dx\, dy \\
{}- 2 \int W(x-y) z_{\varepsilon,\xi}(x) v_1(x) z_{\varepsilon,\xi}(y) v_1(y) dx\, dy \\
= \int_{B(\xi,R+1)} | \nabla v_1|^2 + V(\varepsilon \xi) |v_1|^2 + \int_{B(\xi,R+1)} \left[ V(\varepsilon x) - V(\varepsilon \xi) \right] |v_1|^2 \\
{} - 2 \int W(x-y) z_{\varepsilon,\xi}(x) v_1(x) z_{\varepsilon,\xi}(y) v_1(y) dx\, dy \\
{}-\int W(x-y) z_{\varepsilon,\xi}(x)^2 v_1(y)^2 \, dx\, dy.
\end{multline*}
By the exponential decay of $z_{\varepsilon,\xi}$ and the Hardy--Littlewood--Sobolev inequality, it is easy to show that
\begin{multline*}
\langle \nabla^2 I_\varepsilon (z_{\varepsilon,\xi})v_1,v_1 \rangle = 
\int | \nabla v_1|^2 + V(\varepsilon \xi) |v_1|^2  +o_\varepsilon (1) \\
{} - 2 \int W(x-y) z_{\varepsilon,\xi}(x) v_1(x) z_{\varepsilon,\xi}(y) v_1(y) dx\, dy \\
{}-\int W(x-y) z_{\varepsilon,\xi}(x)^2 v_1(y)^2 \, dx\, dy.
\end{multline*}
It is very important, at this stage, to highlight that the convolution terms play no r\^ole, since they do not contain any potential function. This is the main difference with respect to \cite{AMR}, where the nonlinearity has a potential function $K$ in front. 

The last estimate tells us that, up to an error of size $\sqrt{\varepsilon}$, the estimate of $\langle \nabla^2 I_\varepsilon (z_{\varepsilon,\xi})v_1,v_1 \rangle$ follows from the estimate of the quadratic form
\begin{multline*}
v_1 \mapsto  
\int | \nabla v_1|^2 + V(\varepsilon \xi) |v_1|^2   \\
{} - 2 \int W(x-y) z_{\varepsilon,\xi}(x) v_1(x) z_{\varepsilon,\xi}(y) v_1(y) dx\, dy \\
{}-\int W(x-y) z_{\varepsilon,\xi}(x)^2 v_1(y)^2 \, dx\, dy.
\end{multline*}
\emph{with constant coefficient} $V(\varepsilon \xi)$. But this is the second derivative (along the direction $v_1$) of the Euler--Lagrange functional associated to the limiting equation (\ref{eq:16}). We can repeat the same arguments of \cite[Lemma 5]{AMR}, using of course the non-degeneracy for the linearized limiting problem given by Theorem \ref{th:2} (and Remark \ref{rem:9}) instead of the non-degeneracy for the linearized Schr\"{o}dinger equation, and conclude that
\begin{equation} \label{eq:25}
\langle \nabla^2 I_\varepsilon (z_{\varepsilon,\xi})v_1,v_1 \rangle \geq C \|v_1\|^2 + o_\varepsilon (1).
\end{equation}
Collecting (\ref{eq:23}), (\ref{eq:24}) and (\ref{eq:25}), and recalling that
\[
\|v_1\|^2 + \|v_2\|^2 = 1 + O(\sqrt{\varepsilon}) = \|v\|^2 + O(\sqrt{\varepsilon}),
\]
we prove (\ref{eq:21}).
\end{proof}
\begin{remark}
The non--degeneracy of the limiting problem is used in an essential way to show (\ref{eq:25}).
\end{remark}
The auxiliary equation (\ref{eq:17}) is formally equivalent to a fixed-point equation for the operator
\begin{equation}
\label{ eq:26}
S_\varepsilon (w)=w-\left( P \nabla^2 I_\varepsilon (z_{\varepsilon,\xi}) \right)^{-1} \left( P \nabla I_\varepsilon (z_{\varepsilon,\xi}+w) \right).
\end{equation}
To find $w$ such that $w=S_\varepsilon (w)$, we will construct a function set on which $S_\varepsilon$ is a contractive map.

\subsection{The contraction principle for \(S_\varepsilon\)}

As in \cite{AMR}, it is convenient to restrict the operator $S_\ge$ to a proper subspace of the 
orthogonal space $\left( T_{z_{\varepsilon,\xi}} Z \right)^\perp$, so that we can get precise estimates on its fixed points. We therefore need to construct a suitable subspace of $\left( T_{z_{\varepsilon,\xi}} Z \right)^\perp$ that is invariant under $S_\varepsilon$, and on which $S_\varepsilon$ is contractive.

The idea is to begin with radial solutions of the problem
\begin{equation} \label{eq:LR}
\left\{
\begin{array}{ll}
-\Delta u + m |x|^{-\alpha} u = f(|x|), &|x|>R \\
u(x)=1, &|x|=R \\
u(x) \to 0, &|x| \to +\infty,
\end{array}
\right.
\end{equation}
where $R>0$, $m>0$, $\alpha \in (0,2]$, and $f \colon (R,+\infty) \to \mathbb{R}$ is positive, and decays at infinity. Roughly speaking, the reason why (\ref{eq:LR}) is related to our fixed-point argument is that the qualitative behavior of the decaying potential $V$ is essentially that of $m|x|^{-\alpha}$, as described precisely by Lemma \ref{lem:1}. Looking for radially symmetric solution of (\ref{eq:LR}), and performing the change of variable $v(r)=u(r)r$, we reduce to the problem
\begin{equation} \label{eq:LR1}
\left\{
\begin{array}{ll}
-v''(r) + m r^{-\alpha} v(r) = f(r)r &(r>R) \\
v(R)=R &(r=R) \\
v(r)r^{-1} \to 0 &(r \to +\infty).
\end{array}
\right.
\end{equation}
We highlight that this is precisely equation (21) of \cite{AMR}, where a term has disappeared since in our work $n=3$. The homogeneous equation corresponding to (\ref{eq:LR1}) has a general solution spanned by two generators $v_1$ and $v_2$. Hence, the general solution of (\ref{eq:LR}) is spanned by $u_1(r)=r^{-1} v_1(r)$ and $u_2(r)=r^{-1} v_2(r)$. 

For the reader's sake, we recall some more properties of $v_1$ and $v_2$, which we borrow from \cite{AMR}; we recall that we are working in $\mathbb{R}^3$, so that some formulas become more explicit. In case $\alpha <2$, $v_1$ and $v_2$ are modified Bessel functions. If we denote by $B_\ell^I$ the modified Bessel function of first order, and by $B_\ell^K$ the modified Bessel function of second order, so that they both solve the ordinary differential equation
\[
r^2 y''(r)+r y'(r)-(r^2+\ell^2)y(r)=0,
\]
then
\begin{eqnarray*}
v_1(r)&=& \sqrt{r} \ B_\ell^K \left( \frac{2 \sqrt{m}}{2-\alpha} r^{(2-\alpha)/2} \right), \quad \ell=\frac{1}{2-\alpha} \\
v_2(r)&=& \sqrt{r} \ B_\ell^I \left( \frac{2 \sqrt{m}}{2-\alpha} r^{(2-\alpha)/2} \right), \quad \ell=\frac{1}{2-\alpha}. \\
\end{eqnarray*}
Hence, standard asymptotic properties of the Bessel functions imply that
\begin{eqnarray*}
v_1(r) &\approx& r^{\alpha/4} \exp \left( -\frac{2 \sqrt{m}}{2-\alpha} r^{\frac{2-\alpha}{4}} \right) \\
v_2 (r) &\approx& r^{\alpha/4} \exp \left( \frac{2 \sqrt{m}}{2-\alpha} r^{\frac{2-\alpha}{4}} \right)
\end{eqnarray*}
As a consequence,
\[
u_1(r) \approx r^{\frac{\alpha}{4}-1} \exp \left( -\frac{2 \sqrt{m}}{2-\alpha} r^{\frac{2-\alpha}{4}} \right)
\]
If $\alpha =2$, then
\begin{eqnarray*}
v_1(r) &=& r^{\frac{1-\sqrt{1+4m}}{2}} \\
v_2(r) &=& r^{\frac{1+\sqrt{1+4m}}{2}} \\
\end{eqnarray*}
and therefore
\begin{eqnarray*}
u_1(r) &=& r^{\frac{-1-\sqrt{1+4m}}{2}} \\
u_2(r) &=& r^{\frac{-1+\sqrt{1+4m}}{2}} \\
\end{eqnarray*}
We recall, for future reference, the following result, see \cite{AMR}.

\begin{lemma} \label{lem:16}
Let $u_1$, $u_2$ be defined as above, and let $\bar{\varphi}$ be a solution of (\ref{eq:LR}), where $f \colon (R,+\infty) \to \mathbb{R}$ is a positive, continuous function satisfying the integrability condition
\begin{equation*}
\int_R^{+\infty} r^{2} f(r) u_2(r) \, d r < +\infty.
\end{equation*}
Then there exists a constant $\gamma(R)>0$ such that $\bar{\varphi}(r) \leq \gamma(R) u_1(r)$ for all $r>R$.
\end{lemma}

\bigskip

For any $R>0$ fixed, we define the set $\mathcal{W}_\varepsilon (R)$ as the set of all functions $w \in E$ such that
\begin{equation*}
w(x+\xi) \leq 
\begin{cases}
\gamma(R) \sqrt{\varepsilon} \, u_1(|x|), &\mbox{if $|x| \geq R$} \\
\sqrt{\varepsilon}, &\mbox{if $|x| \leq R$},
\end{cases}
\end{equation*}
where $u_1$ is the function we have introduced above, and $\gamma(R)$ is the constant of Lemma \ref{lem:16}. Finally, we set
\begin{equation*}
\Gamma_\varepsilon (R) = \left\{
w \in E \mid \|w\|_E \leq c_0 \varepsilon , \ w \in \mathcal{W}_\varepsilon (R) \cap \left( T_{z_{\varepsilon,\xi}} Z \right)^\perp
\right\}.
\end{equation*}
We will choose a good value for the constant $c_0>0$ later (see the proof of Lemma \ref{lem:18}).

\begin{remark}
It is possible to select the constants in the truncation $T_\varepsilon$ so that
\[
\left| z_{\varepsilon,\xi}(x)+w(x) \right| < \frac{\bar{c}}{\left( 1+|\varepsilon x| \right)^\theta}
\]
for all $x \in \mathbb{R}^3$ and all $w \in \Gamma_\varepsilon (R)$.
\end{remark}
By the last remark, we are authorized to write
\begin{multline*}
I_\varepsilon (z_{\varepsilon,\xi}+w) = \frac{1}{2} \|z_{\varepsilon,\xi}+w \|_E^2 \\
{}- \frac{1}{4} \int W(x-y) |z_{\varepsilon,\xi}(x)+w(x)|^2 |z_{\varepsilon,\xi}(y)+w(y)|^2 \, dx \, dy.
\end{multline*}
Since $w \perp T_{z_{\varepsilon,\xi}} Z$, it is known that any critical point of $I_\varepsilon$ constrained to the manifold $\{z_{\varepsilon,\xi}+w \mid w \in \Gamma_\varepsilon (R) \}$ is a solution of (\ref{eq:5}). See \cite{AM,AMS} for a proof of this general fact.

We fix a value $\rho>0$ such that (see Lemma \ref{lem:1})
\begin{equation*}
V(\varepsilon x + \varepsilon \xi) \geq m |x|^{-\alpha} \quad\mbox{for $|x| \geq \rho$}.
\end{equation*}

\begin{lemma} \label{lem:18}
For $c_0$ large enough, and $\varepsilon$ sufficiently small, we have 
\[
\|S_\varepsilon (w)\| \leq c_0 \varepsilon,
\]
for all $w \in \Gamma_\varepsilon (\rho)$. In addition, $S_\varepsilon$ is a contraction.
\end{lemma}
\begin{proof}
Let $\bar{C}$ be the constant given in Lemma \ref{lem:10}. Moreover, 
\[
\| (P \nabla^2 I_\varepsilon (z_{\varepsilon,\xi}) )^{-1} \| \leq C'
\]
 for some $C'>0$. Choose $c_0 = 2 C' \bar{C}$. For $w \in \Gamma_\varepsilon (\rho)$, we have
 \[
 \nabla S_\varepsilon (w)\colon v  \mapsto v - \left( P \nabla^2 I_\varepsilon (z_{\varepsilon,\xi}) \right)^{-1} \left( P \nabla^2 I_\varepsilon (z_{\varepsilon,\xi}+w)(v) \right) .
 \]
We apply $P \nabla^2 I_\varepsilon (z_{\varepsilon,\xi})$ and find
\begin{equation*}
\| P \nabla^2 I_\varepsilon (z_{\varepsilon,\xi})(\langle \nabla S_\varepsilon (w),v \rangle )\| =
\| P \nabla^2 I_\varepsilon (z_{\varepsilon,\xi})(v) - P \nabla^2 I_\varepsilon (z_{\varepsilon,\xi}+w)(v) \|.
\end{equation*}
But, for any $w_1$, $w_2 \in \Gamma_\varepsilon(\rho)$,
\begin{multline*}
D^2 I_\varepsilon (z_{\varepsilon,\xi})(w_1,w_2) - D^2 I_\varepsilon (z_{\varepsilon,\xi}+w)(w_1,w_2) = \\
-2 \int W(x-y) z_{\varepsilon,\xi}(x) w_1(x) z_{\varepsilon,\xi}(y)w_2(y) \, dx\, dy \\
{}- \int W(x-y) z_{\varepsilon,\xi}(x)^2 w_1(y)w_2(y) \, dx\, dy \\
{}+2 \int W(x-y) \left( z_{\varepsilon,\xi}(x)+w(x) \right) w_1(x)  \left( z_{\varepsilon,\xi}(y)+w(y) \right) w_2(y) \, dx \, dy \\
{}+ \int W(x-y) \left( z_{\varepsilon,\xi}(x)+w(x) \right)^2 w_1(y) w_2(y) \, dx\, dy \\
= 2 \int W(x-y) w(x) w_1(x) w(y) w_2(y) \, dx\, dy \\
{}+ \int W(x-y) \left( \left( z_{\varepsilon,\xi}(x)+w(x) \right)^2-z_{\varepsilon,\xi}(x)^2 \right) w_1(y) w_2 (y) \, dx\, dy.
\end{multline*}
The first term is estimated by
\begin{equation*}
\left| \int W(x-y) w(x) w_1(x) w(y) w_2(y) \, dx\, dy \right| \leq c_1 \|w_1\|_{L^6} \|w_2\|_{L^6} \|w\|_E^{\frac{8}{3q}},
\end{equation*}
with $q>8/3$. Indeed, we apply the usual Hardy--Littlewood--Sobolev inequality, and we get
\begin{equation*}
\left| \int W(x-y) w(x) w_1(x) w(y) w_2(y) \, dx\, dy \right| \leq \|w \cdot w_1\|_{L^{\frac{6}{5}}} \|w \cdot w_2\|_{L^{\frac{6}{5}}}.
\end{equation*}
Now,
\begin{equation*}
\int |w|^{6/5} |w_1|^{6/5} \leq \left( \int |w_1|^6 \right)^{1/5} \left( \int |w|^{3/2} \right)^{4/5}.
\end{equation*}
For any $q>1$, let $q'=q/(q-1)$. Then
\begin{equation*}
\int |w|^{3/2} \leq \left( \int |w|^{\frac{3}{4}q} \right)^{1/q} \left( \int |w|^{\frac{3}{4}q'} \right)^{1/q'}
\end{equation*}
Now, $\int |w|^{\frac{3}{4}q'}$ is finite, since $w \in \Gamma_\varepsilon (\rho)$. Choose now $q > 8/3$ (so that $\tau = 3q/4 >2$); there exists a constant $c_2$ such that $|w(x)|^{\tau -2} \leq c_2 V(\varepsilon x)$. Hence
\begin{eqnarray*}
\left( \int |w|^{\frac{3}{4}q} \right)^{1/q} &\leq& c_3 \left( \int |w|^2 |w|^{\tau -2} \right)^{1/q} \\
&\leq& c_4 \left( \int |w|^2 V(\varepsilon x) \right)^{1/q}  \leq c_5 \|w\|^{2/q}.
\end{eqnarray*}
Similar arguments can be used to estimate the term
\begin{multline*}
\int W(x-y) \left( \left( z_{\varepsilon,\xi}(x)+w(x) \right)^2-z_{\varepsilon,\xi}(x)^2 \right) w_1(y) w_2 (y) \, dx\, dy \\
= \int W(x-y) \left( w(x) (z_{\ge,\xi}(x)+w(x)) \right) w_1(y) w_2 (y) \, dx\, dy
\end{multline*}
Therefore,
\begin{equation*}
\| P \nabla^2 I_\varepsilon (z_{\varepsilon,\xi})(\langle \nabla S_\varepsilon (w),v \rangle)\| \leq c_6 \|w\|_E^{\delta} \|v\|,
\end{equation*}
for some $\delta < 1$.
If $w_1$, $w_2 \in \Gamma_\varepsilon (\rho)$, then
\begin{multline*}
\|S_\varepsilon (w_1)-S_\varepsilon (w_2) \| \leq \| \left( P \nabla^2 I_\varepsilon (z_{\varepsilon,\xi}) \right)^{-1}\| \times \\
\times \| P \nabla^2 I_\varepsilon (z_{\varepsilon,\xi}) \left( S_\varepsilon (w_1)-S_\varepsilon (w_2) \right) \|
\end{multline*}
and by the Intermediate Valued Theorem we find
\begin{equation*}
\|S_\varepsilon (w_1)-S_\varepsilon (w_2) \| \leq c_7 \left( \max_{0 \leq s \leq 1} \|w_2+s(w_1-w_2) \| \right)^{\delta} \|w_1-w_2\|.
\end{equation*}
Since $w_1$ and $w_2$ belong to $\Gamma_\varepsilon (\rho)$, we easily deduce that 
\begin{equation} \label{eq:35}
\|S_\varepsilon (w_1)-S_\varepsilon (w_2) \| = o_\varepsilon (1) \|w_1-w_2\|.
\end{equation}
Hence $S_\varepsilon$ is a contraction on $\Gamma_\varepsilon(\rho)$. Taking $w_1=w$ and $w_0=0$ in (\ref{eq:35}), we get
\[
\|S_\varepsilon (w)-S_\varepsilon (0) \| = o_\varepsilon (1) \|w\|.
\]
On the other hand,
\begin{align*}
\| S_\varepsilon (0) \| &= \Big\| \left( P \nabla^2 I_\varepsilon (z_{\varepsilon,\xi})\right)^{-1} \left( P \nabla I_\varepsilon (z_{\varepsilon,\xi}) \right) \Big\| \\
&\leq C' \| P \nabla I_\varepsilon (z_{\varepsilon,\xi}) \| \leq C' \bar{C} \varepsilon.
\end{align*}
By the triangular inequality,
\begin{align*}
\|S_\varepsilon (w) \| &\leq \|S_\varepsilon (w)-S_\varepsilon (0) \| + \|S_\varepsilon (0)\| 
\leq o_\varepsilon (1) \|w\| + C' \bar{C} \varepsilon \\
&= o_\varepsilon (1) \|w\|  + \tfrac{1}{2} c_0 \varepsilon.
\end{align*}
Since $w \in \Gamma_\varepsilon$, $\|w\|\leq c_0 \varepsilon$; hence $\|S_\varepsilon (w)\| \leq c_0 \varepsilon$ provided $\varepsilon$ is small enough.
\end{proof}

The most important, and probably the most difficult step, is to prove that $S_\varepsilon$ maps $\Gamma_\varepsilon (\rho)$ into $\mathcal{W}_\varepsilon (\rho)$. Our proof follows \cite{AMR}, but we remark that the form of the convolution term can be exploited to cancel several terms.

\begin{lemma}
For all $\varepsilon$ small enough, $S_\varepsilon (\Gamma_\varepsilon (\rho)) \subset\mathcal{W}_\varepsilon (\rho)$. 
\end{lemma}
\begin{proof}
For simplicity, we set $\tilde{w}=S_\ge (w)$. We notice that 
\[
\tilde{w} = w - \left( P \nabla^2 I_\ge (z_{\ge,\xi}) \right)^{-1} P \nabla I_\ge (z_{\ge,\xi}+w),
\]
which we can write as
\[
P \left(
\nabla^2 I_\ge(z_{\ge,\xi})(\tilde{w}-w) + \nabla I_\ge (z_{\ge,\xi}+w)
\right)=0.
\]
Hence $\tilde{w}$ satisfies the identity
\begin{multline*}
-\Delta \tilde{w} + V(\ge x) \tilde{w} - \tilde{w} \int W(x-y) z_{\ge,\xi}(y)^2 \, dy \\
{}- 2 z_{\ge,\xi} \int W(x-y) z_{\ge,\xi}(y) \tilde{w}(y) \, dy \\
= 2 w \int W(x-y) z_{\ge,\xi}(y)w(y)\, dy + \left( z_{\ge,\xi}+w \right) \int W(x-y)w(y)^2 \, dy \\
{}- \left( -\Delta z_{\ge,\xi} + V(\ge x) z_{\ge,\xi} - z_{\ge,\xi} \int W(x-y) z_{\ge,\xi}(y)^2 \, dy 
\right) \\
{}+ \left( -\Delta \dot{z}_{\ge,\xi} + V(\ge ) \dot{z}_{\ge,\xi}
\right) \eta,
\end{multline*}
where 
\[
\eta = 
\left\langle
\nabla^2 I_\ge (z_{\ge,\xi})(\tilde{w}-w)+\nabla I_\ge (z_{\ge,\xi}+w),\dot{z}_{\ge,\xi}
\right\rangle \|\dot{z}_{\ge,\xi}\|^{-2}
\]
and $\dot{z}_{\ge,\xi}$ stands for a linear combination of the derivatives 
\[
\frac{\partial}{\partial \xi_1} {z}_{\ge,\xi}, \frac{\partial}{\partial \xi_2} {z}_{\ge,\xi}, \frac{\partial}{\partial \xi_3} {z}_{\ge,\xi},
\]
related to the projection of the equation $\tilde{w}=S_\ge (w)$ onto $\left( T_{z_{\ge,\xi}} Z \right)^\perp$.

We define $z_0(x)=z_{\ge,\xi}(x+\xi)$, $w_0(x)=w(x+\xi)$, and $\tilde{w}_0(x)=\tilde{w}(x+\xi)$. Let $L_0$ be the linear operator defined by
\begin{multline*}
L_0 v = -\Delta v + V(\ge x + \ge \xi) v - v \int W(x-y) z_0(y)^2 \, dy \\
{}- 2 z_0 \int W(x-y) z_0(y) v(y)\, dy.
\end{multline*}
Therefore $\tilde{w}_0$ verifies
\begin{multline*}
L_0 \tilde{w}_0  = \\
2 w_0 \int W(x-y) z_0(y) w_0(y) \, dy + \left( z_0+w_0 \right) \int W(x-y) w_0(y)^2 \, dy\\
{}- \left( -\Delta z_0 + V(\ge x + \ge \xi) z_0 - z_0 \int W(x-y) z_0(y)^2 \, dy
\right) \\
{}+ \left( -\Delta \dot{z}_0 + V(\ge x + \ge \xi) \dot{z}_0 \right) \eta.
\end{multline*}
If we set
\begin{align*}
g_1 &= 2 w_0 \int W(x-y) z_0(y) w_0(y) \, dy + \left( z_0+w_0 \right) \int W(x-y) w_0(y)^2 \, dy \\
g_2 &= \left( -\Delta \dot{z}_0 + V(\ge x + \ge \xi) \dot{z}_0 \right) \eta \\
g_3 &= -\Delta z_0 + V(\ge x + \ge \xi) z_0 - z_0 \int W(x-y) z_0(y)^2 \, dy,
\end{align*}
and $g_0 = g_1+g_2+g_3$, then $L_0 \tilde{w}_0 = g_0$.

We need to prove that
\begin{equation} \label{eq:38}
|\tilde{w}_0 (x)| \leq \sqrt{\ge} \quad\mbox{for $|x| \leq \rho$}
\end{equation}
and
\begin{equation} \label{eq:39}
|\tilde{w}_0 (x)| \leq \gamma(\rho) \sqrt{\ge} u_1 (|x|) \quad\mbox{for $|x| \geq \rho$}.
\end{equation}
The main ingredients are of the following estimates:
\begin{equation} \label{eq:40}
\|g_0\|_{L^{q/2}(B(0,2\rho)} \leq o_\ge(1) \sqrt{\ge} \quad\text{for $q>3$}
\end{equation}
and
\begin{equation} \label{eq:41}
|g_0(x)| \leq o_\ge(1) \sqrt{\ge} u_1(|x|) \quad \mbox{for $|x| \geq \rho$}.
\end{equation}
Moreoever, $|g_1(x)| \leq c_1 |w_0(x)|$. Indeed, this follows easily from the definition of $g_1$ and from the inequality
\begin{equation*}
\left \| \int W(\cdot -y) z_0(y) w_0(y) \, dy \right\|_\infty \leq c_2 \|z_0 w_0 \|_{L^{3/2}},
\end{equation*}
which is the limit case of (\ref{eq:hls6}) as $s \to +\infty$.
Now (\ref{eq:40}) and (\ref{eq:41}) follow from arguments that are completely analogous to those carried out in \cite{AMR}, pages 338 and 339. We just suggest how (\ref{eq:38}) and (\ref{eq:39}) are deduced from (\ref{eq:40}) and (\ref{eq:41}).

By standard elliptic estimates (see for instance \cite{GT}, Theorem 8.24) we have, for $q>3$,
\begin{equation*}
\|\tilde{w}_0\|_{L^\infty (B(0,\rho))} \leq c_3 \|\tilde{w}_0\|_{L^2 (B(0,2\rho))} + c_4 \|g_0\|_{L^{q/2}(B(0,2 \rho))}.
\end{equation*}
Since $\|z_{\ge,\xi}\| \leq c_{5}$ and by virtue of Lemma \ref{lem:18}, we deduce 
\[
\|\tilde{w}_0\|_{L^2(B(0,2\rho))} \leq c_{6} \ge.
\]
Coupling this with (\ref{eq:40}), we find
\begin{equation*}
\|\tilde{w}_0\|_{L^\infty (B(0,\rho))} \leq c_{7} \ge + o_\ge (1) \sqrt{\ge}.
\end{equation*}
Hence (\ref{eq:38}) follows. If $\alpha<2$, the proof of (\ref{eq:39}) follows as in \cite{AMR}, by comparing $\tilde{w}_0$ with the solution of the problem
\begin{equation*}
\left\lbrace
\begin{array}{ll}
-\Delta \phi + |x|^{-\alpha} \phi = \sqrt{\ge} u_1 , &|x|>\rho \\
\phi(x)=\sqrt{\ge} &|x|=\rho,
\end{array}
\right.
\end{equation*}
exploiting Lemma \ref{lem:16}. A few changes are in order when $\alpha=2$, and we refer to \cite{AMR}.
\end{proof}
We have all the ingredients to state a result about the existence and uniqueness of a fixed point for $S_\ge$ in $\Gamma_\varepsilon (\rho)$. We omit the proof, which is based on standard techniques and can be found in \cite{AM,AMR,AMS}.
\begin{proposition} \label{prop:22}
Under the assumptions of the Theorem \ref{th:main}, for $\varepsilon$ sufficiently small, there exists one and only one fixed point $w \in \Gamma_\ge (\rho)$ of $S_\ge$. Moreover, $w$ is differentiable with respect to $\xi \in \mathbb{R}^3$, and there exists a constant $C>0$ such that
\begin{equation} \label{eq:42}
\|w\| \leq C \ge, \quad \|\nabla_\xi w\| \leq C \ge^\delta,
\end{equation}
where $\delta < 1$ was found in the proof of Lemma \ref{lem:18}.
\end{proposition}

\section{Analysis of the reduced functional}

We recall that the manifold
\begin{equation*}
\tilde{Z} = \left\{ z_{\ge,\xi}+w \mid z_{\ge,\xi} \in Z,\ w \in \Gamma_\ge (\rho) \right\}
\end{equation*}
is a natural constraint for the functional~$I_\ge$: this is a general feature of the perturbation method that we are using, see \cite{AM}. Loosely speaking, we have found $w \perp T_{z_{\ge,\xi}} Z$ such that $\nabla I_\ge (z_{\ge,\xi}+w) \in T_{z_{\ge,\xi}} Z$. If $I_\ge$ has a critical point constrained to $\tilde{Z}$, then for $\ge$ sufficiently small we have $\nabla I_\ge \in T_{z_{\ge,\xi}} Z \cap \left( T_{z_{\ge,\xi}} Z \right)^\perp = \{0\}$.
In addition, $\tilde{Z} \simeq \mathbb{R}^3$, and therefore the \emph{constrained function} given by
\begin{equation*}
\Phi_\ge (\xi) = I_\ge (z_{\ge,\xi}+w)
\end{equation*}
has the property that every critical point $\xi$ produces a solution $z_{\ge,\xi}+w$ of (\ref{eq:5}) for $\ge$ sufficiently small. However, a direct study of $\Phi_\ge$ on $\mathbb{R}^3$ is difficult, and we want to use its Taylor expansion to locate its critical points. 

\begin{proof}[Proof of Theorem \ref{th:main}]
We expand the function $\Phi_\ge$:
\begin{eqnarray*}
\Phi_\ge(\xi) &=& \frac{1}{2} \int |\nabla (z_{\ge,\xi}+w)|^2 + \frac{1}{2} \int V(\ge x) |z_{\ge,\xi}+w|^2 \\
&&{}-\frac{1}{4} \int W(x-y) \left( z_{\ge,\xi}(x)+w(x) \right)^2 \left( z_{\ge,\xi}(y)+w(y) \right)^2 \, dx \, dy \\
&=& \frac{1}{2} \int |\nabla z_{\ge,\xi}|^2 + \frac{1}{2} |\nabla w|^2 + \int \nabla z_{\ge,\xi} \cdot \nabla w + \frac{1}{2} \int V(\ge x) z_{\ge,\xi}^2 \\
&&{}-\frac{1}{4} \int W(x-y) \left( z_{\ge,\xi}(x)+w(x) \right)^2 \left( z_{\ge,\xi}(y)+w(y) \right)^2 \, dx \, dy \\
&=& \frac{1}{2} \|z_{\ge,\xi} \|^2 + \frac{1}{2} \|w\|^2 + \langle z_{\ge,\xi},w \rangle_E \\
&&{}-\frac{1}{4} \int W(x-y) \left( z_{\ge,\xi}(x)+w(x) \right)^2 \left( z_{\ge,\xi}(y)+w(y) \right)^2 \, dx \, dy.
\end{eqnarray*}
Since
\begin{equation*}
-\Delta z_{\ge,\xi} + V(\ge\xi) z_{\ge,\xi} = \left(W*|z_{\ge,\xi}|^2 \right) z_{\ge,\xi},
\end{equation*}
\begin{equation*}
\|z_{\ge,\xi}\|^2 + \int \left( V(\ge \xi) - V(\ge x) \right) z_{\ge,\xi}^2 = \int W(x-y) z_{\ge,\xi}(x)^2 z_{\ge,\xi}(y)^2 \, dx\, dy,
\end{equation*}
\begin{multline*}
\langle z_{\ge,\xi},w \rangle_E + \int \left( V(\ge \xi) - V(\ge x) \right) z_{\ge,\xi} w \\
= \int W(x-y) z_{\ge,\xi}(x)^2 z_{\ge,\xi}(y)w(y) \, dx\, dy
\end{multline*}
we have
\begin{align*}
\Phi_\ge(\xi) &= \frac{1}{2} \int W(x-y) z_{\ge,\xi}(x)^2 z_{\ge,\xi}(y)^2  dx\, dy +  \frac{1}{2} \int \left( V(\ge \xi) - V(\ge x) \right) z_{\ge,\xi}^2\\
& \quad+ \int W(x-y) z_{\ge,\xi}(x)^2 z_{\ge,\xi}(y) w(y) \, dx\, dy + \frac{1}{2} \|w\|^2 \\
&\quad{}- \frac{1}{4} \int W(x-y) \left( z_{\ge,\xi}(x)+w(x) \right)^2  \left( z_{\ge,\xi}(y)+w(y) \right)^2 \, dx\, dy\\
&\quad{}+\int \left( V(\ge \xi) - V(\ge x) \right) z_{\ge,\xi} w \\
&= \frac{1}{4} \int W(x-y) z_{\ge,\xi}(x)^2 z_{\ge,\xi}(y)^2 \, dx \, dy + \mathcal{R}_\ge(w),
\end{align*}
where 
\begin{multline*}
\mathcal{R}_\ge(w) =\frac{1}{2} \int \left( V(\ge \xi) - V(\ge x) \right) z_{\ge,\xi}^2 \\
{}+ \int W(x-y) z_{\ge,\xi}(x)^2 z_{\ge,\xi}(y) w(y) \, dx\, dy  \\
{}- \frac{1}{4} \int W(x-y) \left( z_{\ge,\xi}(x)+w(x) \right)^2  \left( z_{\ge,\xi}(y)+w(y) \right)^2 \, dx\, dy\\
{}+\int \left( V(\ge \xi) - V(\ge x) \right) z_{\ge,\xi} w + \frac{1}{2} \|w\|^2 = o(\ge)
\end{multline*}
by the usual arguments based on the Intermediate Valued Theorem and on the Hardy--Littlewood--Sobolev inequality and by Proposition \ref{prop:22}. The calculations are completely algebraic, and a direct use of the estimates (\ref{eq:42}) proves the statement. We omit the routine calculations.
Inserting the definition
\begin{equation*}
z_{\ge,\xi}(x) = V(\ge\xi) U \left( \sqrt{V(\ge\xi)} \ (x-\xi) \right),
\end{equation*}
into the expansion just obtained gives
\begin{equation}
\Phi_\ge (\xi) = \frac{1}{4} V(\ge\xi)^{\frac{3}{2}} \int W(x-y) U(x)^2 U(y)^2 \, dx\, dy + o(\ge).
\end{equation}
If $x_0$ (we may of course assume that $x_0=0$) is an isolated local mi\-nimum (or maximum) of $V$, then $\tilde{\Phi}_\ge(\xi)=\Phi_\ge(\xi/\ge)$ possesses a critical point $\xi_\ge \sim 0$, and this gives rise to a solution 
\[
u_\ge \left( \frac{x}{\ge} \right) \sim z_{\ge,\xi_\ge} \left( \frac{x-\xi_\ge}{\ge} \right)
\]
concentrating at $0=x_0$.
\end{proof}

\begin{remark}
Since $z_{\ge,\xi}$ is positive, and $w=w(\ge,\xi)$ is of order $\ge$, the solution $u_\ge$ found in Theorem \ref{th:main} is positive.
\end{remark}

\section{Other existence results and comments}

Although we have proved our main theorem under the assumption that $V$ has an isolated minimum or maximum point, some extensions can be proved. Since their proofs require some rather standard modifications (the statement about a compact manifold of non--degenerate critical points was studied in \cite{AMS}, see also \cite{AM,AMR}) of the perturbation technique, we
collect these generalized existence results in the next theorem.
We recall that a stable critical point $x_0$ of a smooth function $V$ is a critical point at which the Leray--Schauder index $\operatorname{ind}(\nabla V,x_0,0)\neq 0$. It is well-known that strict local minimum (and maximum) points are stable critical points.

\begin{theorem}
Theorem \ref{th:main} is also valid provided $x_0$ in an isolated stable critical point of $V$. More generally, if $\Sigma \subset \mathbb{R}^3$ is either a compact set of maxima (or minima) of $V$, or a compact manifold of non--degenerate critical points of $V_{|\Sigma^\perp}$, then the conclusion of Theorem \ref{th:main} is again true.
\end{theorem}

\bigskip

As recalled in the first section, in \cite{CSS} a more involved system with a magnetic field has been recently studied.  It has been known since the paper \cite{CS} that the electromagnetic Schr\"{o}dinger equation can be handled by means of the perturbation method discussed in our paper. However, it generally requires some boundedness assumptions on the potential that generates the magnetic field, and this excludes many fields of physical interest. However, we believe that our Theorem \ref{th:main} holds (and we find solutions such that $|u_\ge|$ concentrates at ``good'' critical points of $V$) for the non-local equation
\begin{equation*}
\left( \frac{\varepsilon}{\mathrm{i}} \nabla - A(x) \right)^2 u + V(x) u = \left( W * |u|^2 \right) u
\end{equation*}
if the following \emph{technical} assumption is satisfied (along with (V1) and (V2)):
\begin{description}
\item[(A1)] $A \colon \mathbb{R}^3 \to \mathbb{R}^3$ is a \emph{bounded} vector potential, with \emph{bounded} derivative: $\sup_{x \in \mathbb{R}^3} |A'(x)| < \infty$.
\end{description}
Of course solutions are complex--valued, but their moduli still concentrate. The action of $A$ is felt only in the phase of these solutions. This is, loosely speaking, a generic feature, as proved in \cite{SS}. To feel the presence of $A$, one usually needs some singularities in the magnetic field, or a very strong symmetry of the solution. See a recent paper by Cingolani \emph{et al.} (\cite{CC}) for the latter case, though in the local case.
We plan to investigate the non-local case in a future paper.

\bigskip

An interesting open (as far as we know) problem is the generalization of our results to any dimension $n \geq 3$  (the convolution kernel $W$ becomes homogeneous of degree $2-n$, of course). Even for a potential $V$ bounded away from its infimum, there are obstacles. Every paper in the literature deals with $n=3$. We cannot say whether this is only a technical obstruction. With respect to the method we have used, the main difficulty is to extend Theorem \ref{th:2}, in particular the non--degeneracy part. We suspect that the ODE proof contained in \cite{WW} can be adapted to any space dimension.

Another interesting question is whether potentials $V$ with a fast decay can be allowed. More precisely, this corresponds to the assumption that $\liminf_{|x| \to +\infty} V(x)|x|^2 = 0$. The only result which deals with such potentials is \cite{MVS}, but, once more, it seems difficult to carry their penalization scheme to our non-local situation. Moreover, as stated in \cite{MVS}, it would be useless to look for solutions near our manifold $Z_\varepsilon$, since one expects solutions to decay polynomially fast at infinity.


      \bibliographystyle{plain}
    \bibliography{Bibliography}

\begin{thebibliography}{10}

\bibitem{ABC}
A.~Ambrosetti, M.~Badiale, and S.~Cingolani.
\newblock Semiclassical states of nonlinear {S}chr\"odinger equations.
\newblock {\em Arch. Rational Mech. Anal.}, 140(3):285--300, 1997.

\bibitem{AFM}
A.~Ambrosetti, V.~Felli, and A.~Malchiodi.
\newblock Ground states of nonlinear {S}chr\"odinger equations with potentials
  vanishing at infinity.
\newblock {\em J. Eur. Math. Soc. (JEMS)}, 7(1):117--144, 2005.

\bibitem{AM}
A.~Ambrosetti and A.~Malchiodi.
\newblock {\em Perturbation methods and semilinear elliptic problems on {${\bf
  R}\sp n$}}, volume 240 of {\em Progress in Mathematics}.
\newblock Birkh\"auser Verlag, Basel, 2006.

\bibitem{AMR}
A.~Ambrosetti, A.~Malchiodi, and D.~Ruiz.
\newblock Bound states of nonlinear {S}chr\"odinger equations with potentials
  vanishing at infinity.
\newblock {\em J. Anal. Math.}, 98:317--348, 2006.

\bibitem{AMS}
A.~Ambrosetti, A.~Malchiodi, and S.~Secchi.
\newblock Multiplicity results for some nonlinear {S}chr\"odinger equations
  with potentials.
\newblock {\em Arch. Ration. Mech. Anal.}, 159(3):253--271, 2001.

\bibitem{AR}
A.~Ambrosetti and D.~Ruiz.
\newblock Radial solutions concentrating on spheres of nonlinear
  {S}chr\"odinger equations with vanishing potentials.
\newblock {\em Proc. Roy. Soc. Edinburgh Sect. A}, 136(5):889--907, 2006.

\bibitem{BVS}
D.~Bonheure and J.~Van~Schaftingen.
\newblock Bound state solutions for a class of nonlinear {S}chr\"odinger
  equations.
\newblock {\em Rev. Mat. Iberoam.}, 24(1):297--351, 2008.

\bibitem{BW}
J.~Byeon and Zh-Q. Wang.
\newblock Standing waves with a critical frequency for nonlinear
  {S}chr\"odinger equations.
\newblock {\em Arch. Ration. Mech. Anal.}, 165(4):295--316, 2002.

\bibitem{BW2}
J.~Byeon and Zh-Q. Wang.
\newblock Spherical semiclassical states of a critical frequency for
  {S}chr\"odinger equations with decaying potentials.
\newblock {\em J. Eur. Math. Soc. (JEMS)}, 8(2):217--228, 2006.

\bibitem{CC}
S.~Cingolani and M.~Clapp.
\newblock Intertwining semiclassical bound states to a nonlinear magnetic
  {S}chr\"{o}dinger equation.
\newblock {\em Nonlinearity}, 22:2309--2331, 2009.

\bibitem{CS}
S.~Cingolani and S.~Secchi.
\newblock Semiclassical limit for nonlinear {S}chr\"odinger equations with
  electromagnetic fields.
\newblock {\em J. Math. Anal. Appl.}, 275(1):108--130, 2002.

\bibitem{CSS}
S.~Cingolani, S.~Secchi, and M.~Squassina.
\newblock Semiclassical limit for {S}chr\"{o}dinger equations with magnetic
  field and {H}artree--like nonlinearity.
\newblock {\em preprint}, 2009.

\bibitem{DPF}
M.~Del~Pino and P.~Felmer.
\newblock Local mountain passes for semilinear elliptic problems in unbounded
  domains.
\newblock {\em Calc. Var. Partial Differential Equations}, 4(2):121--137, 1996.

\bibitem{GT}
D.~Gilbarg and N.~Trudinger.
\newblock {\em Elliptic partial differential equations of second order}.
\newblock Classics in Mathematics. Springer-Verlag, Berlin, 2001.
\newblock Reprint of the 1998 edition.

\bibitem{LL}
E.~Lieb and M.~Loss.
\newblock {\em Analysis}, volume~14 of {\em Graduate Studies in Mathematics}.
\newblock American Mathematical Society, Providence, RI, second edition, 2001.

\bibitem{MN}
M.~Macr\`{\i} and M.~Nolasco.
\newblock Stationary solutions for the non-linear {H}artree equation with a
  slowly varying potential.
\newblock {\em NoDEA}, in press.

\bibitem{MT}
I.M. Moroz and P.~Tod.
\newblock An analytical approach to the {S}chr\"odinger-{N}ewton equations.
\newblock {\em Nonlinearity}, 12(2):201--216, 1999.

\bibitem{MVS}
V.~Moroz and J.~Van~Schaftingen.
\newblock Semiclassical stationary states for nonlinear {S}chr\"{o}dinger
  equations with fast decaying potentials.
\newblock {\em Cal. Var.}, 2009.

\bibitem{P}
R.~Penrose.
\newblock {Q}uantum computation, entanglement and state reduction.
\newblock {\em Philos. Trans. R. Soc. London Ser. A}, 356(1927), 1998.

\bibitem{SS}
S.~Secchi and M.~Squassina.
\newblock {On the location of spikes for the Schr\"odinger equation with
  electromagnetic field.}
\newblock {\em Commun. Contemp. Math.}, 7(2):251--268, 2005.

\bibitem{S}
E.~M. Stein.
\newblock {\em {Singular integrals and differentiability properties of
  functions.}}
\newblock {Princeton, N.J.: Princeton University Press. XIV, 287 p. }, 1970.

\bibitem{WW}
J.~Wei and M.~Winter.
\newblock {S}trongly interacting bumps for the {S}chr\"{o}dinger--{N}ewton
  equations.
\newblock {\em Journal of Mathematical Physics}, 50, 2009.

\end{thebibliography}

   \end{document}